\documentclass[11pt,leqno]{amsart}
\usepackage{amsmath,amsthm,amssymb}
\usepackage{tabularx}
\newcommand{\GF}{{\mathbb F}}

\newcommand{\R}{{\mathbb R}}

\newcommand{\wt}{{\rm wt}}

\DeclareMathOperator{\Harm}{Harm}

\usepackage{color} 

\newtheorem{Thm}{Theorem}[section]
\newtheorem{Lem}[Thm]{Lemma}
\newtheorem{Cor}[Thm]{Corollary}
\newtheorem{Prop}[Thm]{Proposition}

\theoremstyle{definition}
\newtheorem{Def}[Thm]{Definition}
\newtheorem{Rem}[Thm]{Remark}

\newcommand{\RR}{\mathbb{R}}
\newcommand{\ZZ}{\mathbb{Z}}

\newcommand{\NN}{\mathbb{N}}

\newcommand{\FF}{\mathbb{F}}

\makeatletter
\@addtoreset{equation}{section}

\makeatother

\title[A note on the Assmus--Mattson theorem for some binary codes]{
A note on the Assmus--Mattson theorem for some binary codes
}

\author{Tsuyoshi Miezaki}
\address{		Faculty of Science and Engineering, 
		Waseda University, 
		Tokyo 169--8555, Japan
}
\email{miezaki@waseda.jp} 
\author{Hiroyuki Nakasora*}
\thanks{*Corresponding author}
\address{Institute for Promotion of Higher Education, Kobe Gakuin University, Kobe
651--2180, Japan}
\email{nakasora@ge.kobegakuin.ac.jp}

\date{}



\begin{document}

\begin{abstract}

We previously proposed   
 the first nontrivial examples of a code having support $t$-designs for all weights obtained 
from {the Assmus--Mattson theorem} and having support $t'$-designs for some weights with some $t'>t$. 
This suggests  
the  possibility of generalizing the Assmus--Mattson theorem, which is very important in design and coding theory. 
In the present paper, we generalize this example as a strengthening of the Assmus--Mattson theorem 
along this direction. 
As a corollary, 
we provide a new characterization of 
the extended Golay code $\mathcal{G}_{24}$. 
\end{abstract}

\maketitle

\setcounter{section}{+0}
\section{Introduction}

Let $D_{w}$ be the support design of a binary code $C$ for weight $w$ and
\begin{align*}
\delta(C)&:=\max\{t\in \mathbb{N}\mid \forall w, 
D_{w} \mbox{ is a } t\mbox{-design}\},\\ 
s(C)&:=\max\{t\in \mathbb{N}\mid \exists w \mbox{ s.t.~} 
D_{w} \mbox{ is a } t\mbox{-design}\}.
\end{align*}

We note $\delta(C) \leq s(C)$. 
In our previous paper \cite{extremal design2 M-N}, 
we considered the possible occurrence of $\delta(C)<s(C)$.
If $C$ is an extremal Type II code, 
there is no known example of $\delta(C)<s(C)$.
In \cite{MN-tec}, we found the first nontrivial example of a code 
having support $t$-designs for all weights obtained 
from the Assmus--Mattson theorem and 
having support $t'$-designs for some weights with some $t'>t$. 
This suggests  
a possibility to generalize the Assmus--Mattson theorem. 
Herein, we strengthen the Assmus--Mattson theorem 
along this direction. 

One of the motivations for this research is that 
the Assmus--Mattson theorem is one of the most 
important theorems in design and coding theory. 
Assmus--Mattson-type theorems 
in the theory of lattices and vertex operator algebra are 
known as the Venkov theorem and the H\"ohn theorem \cite{{Venkov},{H1}}. 
For example, $E_8$-lattice and moonshine vertex operator algebra 
$V^{\natural}$ provide 
spherical $7$-designs for all $(E_8)_{2m}$ and 
conformal $11$-designs for all $(V^{\natural})_{m}$, $m>0$. 
It is interesting to note that 
the $(E_8)_{2m}$ and $(V^{\natural})_{m+1}$ 
are a spherical $8$-design and a conformal $11$-design if and only if 
$\tau(m)=0$, where 
\[
q\prod_{m=1}^{\infty}(1-q^m)^{24}=\sum_{m=0}^{\infty}\tau(m)q^m, 
\]
and D.H.~Lehmer conjectured in \cite{Lehmer} that 
\[
\tau(m)\neq 0
\]
for all $m$ \cite{{Miezaki},{Venkov},{Venkov2}}. 
Therefore, it is interesting to determine the lattice $L$
(resp.~vertex operator algebra $V$) such that 
$L_m$ (resp.~$V_m$) are spherical (resp.~conformal) $t$-designs 
for all $m$ by the Venkov theorem (resp.~H\"ohn theorem) and 
$L_{m'}$ (resp.~$V_{m'}$) are spherical (resp.~conformal) $t'$-designs 
for some $m'$ with some $t'>t$. 
This work is inspired by these facts. 
Namely, we consider the possible occurrence of $\delta(C)<s(C)$.
For related results, 
see
\cite{{Bannai-Koike-Shinohara-Tagami},{BM1},{BM2},{BMY},{Miezaki2},{extremal design H-M-N},{MMN},{extremal design2 M-N}}.

Here, we explain our main results. 
Throughout this paper, $C$ will denote a 
binary $[n,k,d]$ code and $\mathbf{1}_n\in C$. 
Let $C^{\perp}$ be a binary $[n,n-k,d^{\perp}]$ dual code of $C$, and 
let us set $C_u:=\{c\in C\mid \wt(c)=u\}$. 
Note that $d^{\perp}$ is even since $\mathbf{1}_n\in C$. 
We always assume that there exists $t\in \NN$ that
satisfies the following condition: 
\begin{align}
d^{\perp}-t=\sharp\{u\mid C_u \neq \emptyset, 0<u\leq n-t\}. \label{eqn:AM}
\end{align}
This is a condition of the Assmus--Mattson theorem 
(see Theorem \ref{thm:assmus-mattson}), and say the AM-condition.
Let $D_{u}$ and $D^{\perp}_{w}$ be the support designs of $C$ and $C^{\perp}$ 
for weights $u$ and $w$, respectively. 
Then, by (\ref{eqn:AM}) and Theorem \ref{thm:assmus-mattson}, 
$D_u$ and $D^{\perp}_w$ are $t$-designs 
(also $s$-designs for $0<s<t$)
for any $u$ and $w$, respectively. 

The main results of the present paper are the following theorems. 
Let $C$ be satisfying the AM-condition.
Then, we impose some restrictions for $d^{\perp}$ and $t$.
\begin{Thm}\label{thm:main1}
\begin{enumerate}
\item [{\rm (1)}]
If $C$ is satisfying the AM-condition
with $d^{\perp}-t=1$, 
then $(d^{\perp},t)=(2,1)$ or $(4,3)$.

Moreover, we have $\delta(C)=s(C)=\delta(C^{\perp})=s(C^{\perp})=1$ or $3$.

\item [{\rm (2)}]
If $C$ is satisfying the AM-condition with $d^{\perp}-t=2$, 
then $(d^{\perp},t)=(4,2)$.

\item [{\rm (3)}]
If $C$ is satisfying the AM-condition with $d^{\perp}-t=3$, 
then $(d^{\perp},t)=(4,1), (6,3)$, or $(8,5)$.

\end{enumerate}
\end{Thm}
By Theorem \ref{thm:main1}, 
if $C$ is satisfying the AM-condition with $d^{\perp}-t=1$, 
then $\delta(C)=s(C)=\delta(C^{\perp})=s(C^{\perp})$.
For cases where $d^{\perp}-t=2$ or $3$,
the following theorem gives a criterion for $n$ and $d$ such that 
$\delta(C^{\perp})<s(C^{\perp})$ occurs.
\begin{Thm}\label{thm:main3}
\item [{\rm (1)}]
Let $C$ be satisfying the AM-condition with $(d^{\perp},t)=(4,2)$. 
If the equation
\[ 
\sum_{i=0}^{w}  (-1)^{w-i} \binom{d-3}{w-i}\binom{n-2d}{2i+1}
=0 
\] 
is satisfied, 
then
$D^{\perp}_{2w+4}$ is a $3$-design. 

\item [{\rm (2)}]

Let $C$ be satisfying the AM-condition with $(d^{\perp},t)=(4,1)$ or $(6,3)$. 
If the equation
\[ 
\left(\sum_{i=0}^{w}  (-1)^{w-i} \binom{d-(t+1)}{w-i}\binom{n-2d}{2i}\right)+
(-1)^{w+1} \binom{n/2-(t+1)}{w} 
=0 
\] 
is satisfied, 
then
$D^{\perp}_{2w+t+1}$ is a $(t+1)$-design. 
\end{Thm}
This theorem provides a strengthening of the Assmus--Mattson theorem 
for some particular cases.
In Section~\ref{sec:main3}, 
we discuss the 
parameters $n$ and $d$ 
that satisfy the condition in the Theorem \ref{thm:main3} (2). 
In particular, 
we show the following corollary: 
\begin{Cor}\label{cor:main1}
Let $C$ be satisfying the AM-condition with 
$(d^{\perp},t)=(4,1)$.
For $n\leq 10000$, in Table \ref{tab: B}, 
we give the parameters $n$ and $d$ 
such that $\delta(C)<s(C)$ occurs. 

\end{Cor}

Let $C$ be satisfying the AM-condition with 
$(d^{\perp},t)=(8,5)$. 
Then there is no possibility of $\delta(C^{\perp})<s(C^{\perp})$. 
In fact, $C$ must be the extended Golay code $\mathcal{G}_{24}$. 
\begin{Thm}\label{thm:main4}
Let $C$ be satisfying the AM-condition with 
$(d^{\perp},t)=(8,5)$. 
Then $C$ is the extended Golay code $\mathcal{G}_{24}$. 
\end{Thm}
It is interesting to note that 
this theorem provides a new characterization of 
the extended Golay code $\mathcal{G}_{24}$.


This paper is organized as follows. 
In Section~\ref{sec:pre}, we give background material and terminology.
We review the concept of 
harmonic weight enumerators and some theorems of designs, 
which are used in the proof of the 
main results. 
In Section~\ref{sec:d-t=1,2}, 
we state for cases $d^{\bot}-t=1$ and $2$. 
In Section~\ref{sec:d-t=3}, 
we give the proof of theorem~\ref{thm:main1} (3) for case $d^{\bot}-t=3$.
We prove theorems~\ref{thm:main3} (2) and \ref{thm:main4} in Section~\ref{sec:main3}.
Finally, in Section~\ref{sec:rem}, 
we conclude with some remarks.

All computer calculations in this paper were done with the help of 
Mathematica \cite{Mathematica}.


\section{Preliminaries}\label{sec:pre}

\subsection{Background material and terminology}\label{sec:terminology}

Let $\GF_q$ be the finite field of $q$ elements. 
A binary linear code $C$ of length $n$ is a subspace of $\GF_2^n$. 
An inner product $({x},{y})$ on $\FF_2^n$ is given 
by
\[
(x,y)=\sum_{i=1}^nx_iy_i,
\]
for all $x = (x_1,\ldots,x_n)$ and $y = (y_1,\ldots,y_n) \in \mathbb{F}_2^n$. 
The duality of a linear code $C$ is defined as follows: 
\[
C^{\perp}=\{{y}\in \FF_{2}^{n}\ | \ ({x},{y}) =0\ \mbox{ for all }{x}\in C\}.
\]
A linear code $C$ is self-dual 
if $C=C^{\perp}$. 
For $x \in\FF_2^n$,
the weight $\wt(x)$ is the number of its nonzero components. 
The minimum distance of code $C$ is 
$\min\{\wt( x)\mid  x \in C,  x \neq  0 \}$. 
A linear code of length $n$, dimension $k$, and 
minimum distance $d$ is called an $[n,k,d]$ code 
(or $[n,k]$ code) and 
the dual code is called an $[n,n-k,d^{\perp}]$ code. 
{

A $t$-$(v,k,{\lambda})$ design (or $t$-design for short) is a pair 
$\mathcal{D}=(X,\mathcal{B})$, where $X$ is a set of points of 
cardinality $v$, and $\mathcal{B}$ is a collection of $k$-element subsets
of $X$ called blocks, with the property that any $t$ points are 
contained in precisely $\lambda$ blocks.

The support of a nonzero vector ${x}:=(x_{1}, \dots, x_{n})$, 
$x_{i} \in \GF_{2} = \{ 0,1 \}$ is 
the set of indices of its nonzero coordinates: ${\rm supp} ({ x}) = \{ i \mid x_{i} \neq 0 \}$\index{$supp (x)$}. 
The support design of a code of length $n$ for a given nonzero weight $w$ is the design 
with points $n$ of coordinate indices and whose blocks are the supports of all codewords of weight $w$.

The following theorem is from Assmus and Mattson \cite{assmus-mattson}. It is one of the 
most important theorems in coding theory and design theory:
\begin{Thm}[\cite{assmus-mattson}] \label{thm:assmus-mattson}
Let $C$ be a binary $[n,k,d]$ linear code and $C^{\bot}$ be the $[n,n-k,d^{\bot}]$ dual code. 
Let $t$ be an integer less than $d$.  
Let $C$ have at most $d^{\bot}-t$ non-zero weights less than or equal to $n-t$. Then, 
for each weight $u$ with $d \leq u \leq n-t$, the support design in $C$ is a $t$-design, 
and for each weight $w$ with $d^{\bot} \leq w \leq n$, the support design 
in $C^{\bot}$ is a $t$-design.

\end{Thm}

\subsection{The harmonic weight enumerators}\label{sec:weight enumerators}

In this section, we review the concept of 
harmonic weight enumerators. 

Let $C$ be a code of length $n$. The weight distribution of code $C$ 
is the sequence $\{A_{i}\mid i=0,1, \dots, n \}$, where $A_{i}$ is the number of codewords of weight $i$. 
The polynomial
$$W_C(x ,y) = \sum^{n}_{i=0} A_{i} x^{n-i} y^{i}$$
is called the weight enumerator of $C$.
The weight enumerator of code $C$ and its dual $C^{\perp}$ are related. 
The following theorem, proposed by MacWilliams, is called the MacWilliams identity:
\begin{Thm}[\cite{mac}]\label{thm: macwilliams iden.} 
Let $W_C(x ,y)$ be the weight enumerator of an $[n,k]$ code $C$ over $\GF_{q}$ and let $W_{C^\perp}(x ,y)$ be 
the weight enumerator of the dual code $C^\perp$. Then
$$W_{C^\perp} (x ,y)= q^{-k} W_C(x+(q-1)y,x-y).$$
\end{Thm}

A striking generalization of the MacWilliams identity 
was given by Bachoc \cite{Bachoc}, 
who proposed the concept of harmonic weight enumerators and 
a generalization of the MacWilliams identity. 
The harmonic weight enumerators have many applications; 
in particular, the relations between coding theory and 
design theory are reinterpreted and progressed by the harmonic weight 
enumerators \cite {Bachoc,Bannai-Koike-Shinohara-Tagami}. 
For the reader's convenience, we quote 
the definitions and properties of discrete harmonic functions from \cite{Bachoc,Delsarte}. 

Let $\Omega=\{1, 2,\ldots,n\}$ be a finite set (which will be the set of coordinates of the code) and 
let $X$ be the set of its subsets, while for all $k= 0,1, \ldots, n$, $X_{k}$ is the set of its $k$-subsets.
We denote by $\R X$, $\R X_k$ the free real vector spaces spanned by, respectively, the elements of $X$, $X_{k}$. 
An element of $\R X_k$ is denoted by
$$f=\sum_{z\in X_k}f(z)z$$
and is identified with the real-valued function on $X_{k}$ given by 
$z \mapsto f(z)$. 

Such an element $f\in \R X_k$ can be extended to an element $\widetilde{f}\in \R X$ by setting, for all $u \in X$,
$$\widetilde{f}(u)=\sum_{z\in X_k, z\subset u}f(z).$$
If an element $g \in \R X$ is equal to some $\widetilde{f}$, for $f \in \R X_{k}$, we say that $g$ has degree $k$. 
The derivative $\gamma$ is the operator defined by linearity from 
$$\gamma(z) =\sum_{y\in X_{k-1},y\subset z}y$$
for all $z\in X_k$ and for all $k=0,1, \ldots n$, and $\Harm_{k}$ is the kernel of $\gamma$:
$$\Harm_k =\ker(\gamma|_{\R X_k}).$$

\begin{Thm}[{{\cite[Theorem 7]{Delsarte}}}]\label{thm:design}
A set $\mathcal{B} \subset X_{m}$, where $m \leq n$ of blocks is a $t$-design 
if and only if $\sum_{b\in \mathcal{B}}\widetilde{f}(b)=0$ 
for all $f\in \Harm_k$, $1\leq k\leq t$. 
\end{Thm}

In \cite{Bachoc}, the harmonic weight enumerator associated with a binary linear code $C$ was defined as follows: 
\begin{Def}
Let $C$ be a binary code of length $n$ and let $f\in\Harm_{k}$. 
The harmonic weight enumerator associated with $C$ and $f$ is

$$W_{C,f}(x,y)=\sum_{{c}\in C}\widetilde{f}({c})x^{n-\wt({c})}y^{\wt({c})}.$$
\end{Def}

Bachoc proved the following MacWilliams-type equality: 
\begin{Thm}[\cite{Bachoc}] \label{thm: Bachoc iden.} 
Let $W_{C,f}(x,y)$ be 
the harmonic weight enumerator associated with the code $C$ 
and the harmonic function $f$ of degree $k$. Then 
$$W_{C,f}(x,y)= (xy)^{k} Z_{C,f}(x,y)$$
where $Z_{C,f}$ is a homogeneous polynomial of degree $n-2k$, and satisfies
$$Z_{C^{\bot},f}(x,y)= (-1)^{k} \frac{2^{n/2}}{|C|} Z_{C,f} \left( \frac{x+y}{\sqrt{2}}, \frac{x-y}{\sqrt{2}} \right).$$
\end{Thm}

\subsection{Some theorems of designs}\label{sec:design}

Herein, 
we give some theorems that will be used in the proof of the main theorems. 

Let $\mathcal{D}=(X,\mathcal{B})$ be a $t$-$(v,k,\lambda)$ design.
The derived design $\mathcal{D}_{p}$ of $\mathcal{D}$ with respect to the point $p \in X$ 
has point set $X \setminus \{ p \}$ and block set $\{ B \setminus \{ p \} \mid B \in \mathcal{B}, p \in B \}$.
It is a $(t-1)$-$(v-1,k-1,\lambda)$ design.
The complementary design of $\mathcal{D}$ is $\overline{\mathcal{D}}=(X, \overline{\mathcal{B}})$, 
where $\overline{\mathcal{B}}= \{ X \setminus B \mid B \in \mathcal{B} \}$.
If $\mathcal{D}= \overline{\mathcal{D}}$, $\mathcal{D}$ is called a self-complementary design.
Let $D_{n/2}$ be the support $t$-design of the middle weight of 
a code of length $n$. 
It is easily seen that $D_{n/2}$ is self-complementary.

Alltop \cite{alltop} proved the following theorem.
\begin{Thm}[\cite{alltop}]\label{thm:alltop}
If $\mathcal{D}$ is a $t$-design for an even integer $t$ and 
is self-complementary, 
then $\mathcal{D}$ is also a $(t+1)$-design.
\end{Thm}

The block intersection numbers of $\mathcal{D}$ are 
the numbers of points contained in any two blocks of $\mathcal{D}$.
The following theorem is obtained by known results, for example, see 
\cite[Theorem 1.15, Theorem 1.52, Theorem 1.54, Proposition 5.7]{CL}.

\begin{Thm}\label{thm:S=1,2,3}
Let $S$
be the block intersection numbers of $\mathcal{D}$.

\begin{enumerate}
\item[$(1)$] If $S=1$, then $t \leq 2$, with equality if and only if $\mathcal{D}$ is a symmetric design.
\item[$(2)$] If $S=2$, then $\mathcal{D}$ is a quasi-symmetric design and $t \leq 4$, 
with equality if and only if $\mathcal{D}$ is the unique $4$-$(23,7,1)$ design.

\end{enumerate}
\end{Thm}

\section{Case $d^{\bot}-t=1,2$}\label{sec:d-t=1,2}
In this section, 
we give the proof of Theorem \ref{thm:main1} (1), (2) and Theorem \ref{thm:main3} (1).

\subsection{Proof of Theorem \ref{thm:main1} (1), Case $d^{\perp}-t=1$}

Here, 
we provide the proof of Theorem \ref{thm:main1} (1); 
therefore, we always assume that $d^{\perp}-t=1$. 
Then the weight distribution of $C$ is $0,n/2,n$ and $n$ is even. 

\begin{Lem}\label{lem:1 or 3}
If $C$ is satisfying the AM-condition,
then $D_{n/2}$ is a $1$ or $3$-design.
\end{Lem}

\begin{proof}
We state that $D_{n/2}$ is a self-complementary design.
By Theorem \ref{thm:alltop}, $D_{n/2}$ is a $t$-design with an odd number $t$.
Let $S$ be the block intersection numbers of $D_{n/2}$. 
Then we have $S=1$ or 2. 
By Theorem \ref{thm:S=1,2,3} (1) and (2), $D_{n/2}$ is a $t$-design with $t \leq 3$.
Hence, we have $t=1$ or $3$.
\end{proof}

\begin{Lem}\label{lem:not t+1}
If $C$ is satisfying the AM-condition and 
$D_{n/2}$ is not a $(t+1)$-design, 
then $(D^{\perp})_{w}$ is not a $(t+1)$-design for all $w$. 
\end{Lem}

\begin{proof}
The harmonic weight enumerator of $f\in \Harm_{t+1}$ is 
\begin{align*}
W_{C,f}
&=ax^{\frac{n}{2}}y^{\frac{n}{2}}\\
&=(xy)^{t+1}ax^{\frac{n}{2}-(t+1)}y^{\frac{n}{2}-(t+1)}. 
\end{align*}
for some $a\in \RR$ and $a \neq 0$. 
Set
\begin{align*}
Z_{C,f}:=x^{\frac{n}{2}-(t+1)}y^{\frac{n}{2}-(t+1)}. 
\end{align*}
Then by Theorem \ref{thm: Bachoc iden.}, 
\begin{align*}
Z_{C^{\perp},f}
&=a'(x+y)^{\frac{n}{2}-(t+1)}(x-y)^{\frac{n}{2}-(t+1)}, 
\end{align*}
for some $a'\in \RR$ and $a' \neq 0$. 
Then the coefficient of $x^{2(\frac{n}{2}-(t+1)-2i)}y^{2i}$ is not zero. 
By Theorem \ref{thm:design}, 
$(D^{\perp})_{w}$ is not a $(t+1)$-design for all $w$. 

\end{proof}

By Lemma \ref{lem:1 or 3} and \ref{lem:not t+1}, 
we have Theorem \ref{thm:main1} (1).




\subsection{Proof of Theorem \ref{thm:main1} (2) and Theorem \ref{thm:main3} (1), Case $d^{\perp}-t=2$}

We always assume that $d^{\perp}-t=2$. 
Then the weight distribution of $C$ is $0,d, n-d, n$.
We provide the proof of Theorem \ref{thm:main1} (2).


\begin{proof}
Let $S$ be the block intersection numbers of $D_{d}$. 
Then we have $S=1$ or 2. 
By Theorem \ref{thm:S=1,2,3} (1) and (2), $D_{d}$ is a $t$-design with $t \leq 3$.
Hence, we have $(d^{\perp},t)=(4,2)$.
\end{proof}

Here, 
we provide the proof of Theorem \ref{thm:main3} (1). 
\begin{proof}
The harmonic weight enumerator of $f\in \Harm_{3}$ is 
\begin{align*}
W_{C,f}
&=ax^{n-d}y^d+bx^{d}y^{n-d}\\
&=(xy)^{3}(ax^{n-d-3}y^{d-3}+bx^{d-3}y^{n-d-3}), 
\end{align*}
where $a,b\in \RR$. 
Set
\begin{align*}
Z_{C,f}=ax^{n-d-3}y^{d-3}+bx^{d-3}y^{n-d-3}. 
\end{align*}
Then by Theorem \ref{thm: Bachoc iden.}, 
\begin{align*}
Z_{C^{\perp},f}
=&a'(x+y)^{n-d-3}(x-y)^{d-3}\\
&+b'(x+y)^{d-3}(x-y)^{n-d-3}. 
\end{align*}
Since $d^{\perp}=4$, the coefficient of $x^{n-6}$ is zero. 
Hence, we have $a'+b'=0$. 
Then 
\begin{align}
Z_{C^{\perp},f}
=&a'(x+y)^{n-d-3}(x-y)^{d-3} -a'(x+y)^{d-3}(x-y)^{n-d-3} \notag\\
=&a'\big( (x+y)^{n-d-3}(x-y)^{d-3}-(x+y)^{d-3}(x-y)^{n-d-3} \big) \notag \\
=&a'(x^{2}-y^{2})^{d-3} \left((x+y)^{n-2d}- (x-y)^{n-2d} \right) \label{eqn:w2} 
\end{align}

Let 
\[
W_{C^\perp,f}=(xy)^{3}Z_{C^{\perp},f}=\sum{a_i}x^{n-i}y^i.
\]
By (\ref{eqn:w2}), for some $c\in \RR$, 
\begin{align*}
a_{2w+4}=c\times 
\left(\sum_{i=0}^{w}  (-1)^{w-i} \binom{d-3}{w-i}\binom{n-2d}{2i+1}\right).
\end{align*}
By Theorem \ref{thm:design}, 
if the equation 
\[ 
\sum_{i=0}^{w}  (-1)^{w-i} \binom{d-3}{w-i}\binom{n-2d}{2i+1}=0, 
\]
$D^{\perp}_{2w+4}$ is a $3$-design. 
\end{proof}

\section{Proof of Theorem \ref{thm:main1} (3), Case $d^{\perp}-t=3$}\label{sec:d-t=3}

In this section, we provide the proof of Theorem \ref{thm:main1} (3),
therefore,  we always assume that $d^{\perp}-t=3$. 
Then the weight distribution of $C$ is $0,d,n/2,n-d,n$ and $n$ is even.



Let} 
\[
W_{C} (x,y)=x^{n}+\alpha x^{n-d}y^{d}+\beta x^{\frac{n}{2}}y^{\frac{n}{2}}+\alpha x^{d}y^{n-d}+y^{n}
\]
be the weight enumerator of $C$. 
Since $\dim C=k$,
\begin{equation}\label{eq:dim k}
2\alpha+\beta+2=2^{k}.
\end{equation}
First, we show that 
if $d^\perp>8$ then 
we have the following constraint equations, (\ref{eq:y2})--(\ref{eq:y8}).
By Theorem \ref{thm: macwilliams iden.},
\begin{align*}
W_{C^\perp} (x ,y)=& 2^{-k} W_C(x+y,x-y) \\
=& 2^{-k} \big( (x+y)^{n}+\alpha (x+y)^{n-d}(x-y)^{d}+\beta (x+y)^{\frac{n}{2}}(x-y)^{\frac{n}{2}}\\
&+\alpha (x+y)^{d}(x-y)^{n-d}+(x-y)^{n} \big)\\
=& 2^{-k} \big( (x+y)^{n}+\alpha (x^{2}-y^{2})^{d}(x+y)^{n-2d}+\beta (x^{2}-y^{2})^{\frac{n}{2}}\\
&+\alpha (x^{2}-y^{2})^{d}(x-y)^{n-2d}+(x-y)^{n} \big).
\end{align*}
If the coefficient of $x^{n-2i}y^{2i}$ in $W_{C^\perp} (x ,y)$ is zero, 
then we have 
\begin{align}
&2\alpha \left(-d+\binom{n-2d}{2} \right)-\beta\frac{n}{2}+2 \binom{n}{2}=0
\mbox{ for $i=1$},
\label{eq:y2}\\
&2\alpha \left( \binom{d}{2}-d\binom{n-2d}{2}+ \binom{n-2d}{4} \right)+\beta\binom{n/2}{2}+2 \binom{n}{4}=0
\mbox{ for $i=2$},
\label{eq:y4}\\
&2\alpha \left(-\binom{d}{3}+ \binom{d}{2}\binom{n-2d}{2}-d\binom{n-2d}{4}+ \binom{n-2d}{6} \right)\label{eq:y6}\\
&\hspace{10pt}-\beta\binom{n/2}{3}+2 \binom{n}{6}=0\nonumber
\mbox{ for $i=3$}, 
\\
&2\alpha \left(\binom{d}{4}-\binom{d}{3}\binom{n-2d}{2}+ \binom{d}{2}\binom{n-2d}{4}-d\binom{n-2d}{6}+ \binom{n-2d}{8} \right)\label{eq:y8}\\
&\hspace{10pt}+\beta\binom{n/2}{4}+2 \binom{n}{8}=0\nonumber
\mbox{ for $i=4$}.
\end{align}

Therefore, 
if $d^\perp>8$ then 
we have the constraint equations, (\ref{eq:y2})--(\ref{eq:y8}). 
Using Equations (\ref{eq:dim k})--(\ref{eq:y8}), 
we show the following restriction on $n$ and $d$. 
\begin{Prop}\label{pro:d^perp=8}
If $C$ has $d^{\perp} \geq 8$, then 
\[
n^{2}-(4d+3)n+4d^{2}+8=0.
\]

\end{Prop}

\begin{proof}
Assume that $C$ has $d^{\perp}> 8$. 
Using Equations (\ref{eq:y2}) and (\ref{eq:y4}), 
we delete their constant terms as follows 
\[
\frac{(n-2)(n-3)}{2}\,\mbox{LHS of Eq.}(\ref{eq:y2})-6\,\mbox{LHS of Eq.}(\ref{eq:y4})=0
\Leftrightarrow X_{11}\alpha+X_{12}\beta=0,
\]
where $X_{11}$ (resp.~$X_{12}$) is 
the coefficient of $\alpha$ (resp.~$\beta$). 
Similarly, 
Using Equations (\ref{eq:y2}) and (\ref{eq:y6}), 
we delete their constant terms as follows 
\[
2\binom{n}{6}\,\mbox{LHS of Eq.}(\ref{eq:y2})-n(n-1)\,\mbox{LHS of Eq.}(\ref{eq:y6})=0
\Leftrightarrow X_{21}\alpha+X_{22}\beta=0,
\]
where $X_{21}$ (resp.~$X_{22}$) is 
the coefficient of $\alpha$ (resp.~$\beta$). 
We note that the explicit values of $X_{ij}$ are listed in
Section \ref{sec:Xij}. 

Therefore, if $C$ has $d^{\perp}> 8$, then the linear equations systems
\[
\begin{pmatrix}
X_{11}&X_{12}\\
X_{21}&X_{22}
\end{pmatrix}
\begin{pmatrix}
\alpha\\
\beta
\end{pmatrix}
=
\begin{pmatrix}
0\\
0
\end{pmatrix}.
\]
have a nontrivial solution. 
Hence, we have 
\[
\det\begin{pmatrix}
X_{11}&X_{12}\\
X_{21}&X_{22}
\end{pmatrix}
=0.
\]
By a direct computation, 
we have 
\[
d (n-d) (n-2) (n-1) n^3 (n-2 d)^2 (n^{2}-(4d+3)n+4d^{2}+8)=0
\]
Since $n\neq 0,1,2,d,2d$ and $d\neq0$, we have 
\[
n^{2}-(4d+3)n+4d^{2}+8=0.
\]
\end{proof}

The following theorem is a crucial part of the 
proof of Theorem \ref{thm:main1} (3). 
\begin{Thm}\label{thm:d^perp>8}
There is no code $C$ with $d^{\perp}>8$.
\end{Thm}

\begin{proof}
Assume that $C$ has $d^{\perp}> 8$. 
Using Equations (\ref{eq:y6}) and (\ref{eq:y8}), 
we delete their constant terms as follows 
\[
\binom{n}{8}\,\mbox{LHS of Eq.}(\ref{eq:y6})-\binom{n}{6}\,\mbox{LHS of Eq.}(\ref{eq:y8})=0
\Leftrightarrow X_{31}\alpha+X_{32}\beta=0,
\]
where $X_{31}$ (resp.~$X_{32}$) is 
the coefficient of $\alpha$ (resp.~$\beta$). 
We note that the explicit values of $X_{ij}$ are listed in
Section \ref{sec:Xij}. 

If $C$ has $d^{\perp}> 8$, then the linear equations systems
\[
\begin{pmatrix}
X_{11}&X_{12}\\
X_{31}&X_{32}
\end{pmatrix}
\begin{pmatrix}
\alpha\\
\beta
\end{pmatrix}
=
\begin{pmatrix}
0\\
0
\end{pmatrix}.
\]
have a nontrivial solution. 
Hence, we have 
\[
\det\begin{pmatrix}
X_{11}&X_{12}\\
X_{31}&X_{32}
\end{pmatrix}
=0.
\]
By a direct computation, 
we have 
\begin{align*}
&d (n-d) (n-5) (n-4) (n-3) (n-2)^2 (n-1) n^3 (n-2d)^2 \\
&(n^4-(15 + 8 d) n^3+4 (25 + 3 d (5 + 2 d)) n^2\\
&-4 (60 + d (70 + d (15 + 8 d))) n+8 (53 + 35 d^2 + 2 d^4))=0
\end{align*} 
Since $n\neq 0,1,2,3,4,5,d,2d$ and $d\neq0$, we have 
\begin{align}
&n^4-(15 + 8 d) n^3+4 (25 + 3 d (5 + 2 d)) n^2 \label{eq:n4}\\
&-4 (60 + d (70 + d (15 + 8 d))) n+8 (53 + 35 d^2 + 2 d^4)=0. \notag
\end{align}
By Proposition~\ref{pro:d^perp=8}, we have
\[ n= \frac{4d+3 \pm \sqrt{24d-23}}{2}. \]
Since $24d-23$ must be a square number, we have 
\[ d= \frac{m^{2}+23}{24}, n= \frac{m^{2}+41 \pm 6m}{12}. \]
By a direct computation, Eq.(\ref{eq:n4}) does not have a non-trivial solution. 
Hence we have $d^{\perp}<8$, a contradiction.

\end{proof}


Now, we provide the proof of Theorem \ref{thm:main1} (3). 
\begin{proof}[Proof of Theorem \ref{thm:main1} (3)]
By Theorem \ref{thm:d^perp>8}, we have $t=1,3$, or 5.
Hence, we have $(d^{\perp},t)=(4,1), (6,3)$, or $(8,5)$.
\end{proof}








\section{Proofs of Theorems \ref{thm:main3} (2), \ref{thm:main4}
and Corollary \ref{cor:main1}}\label{sec:main3}

\subsection{Constraints on $n$ and $d$}

Let $C$ be satisfying the AM-condition. 
Here, 
we show the following lemma, which sets 
constraints on $n$, $k$, and $d$.
This lemma will be used in the proofs of 
Corollary \ref{cor:main1} and Theorem \ref{thm:main4}. 
\begin{Lem}\label{lem:d^perp=6}
\begin{enumerate}
\item[$(1)$] 
If $C$ has $d^{\perp} \geq 4$ and $\dim C=k$ $(2 \leq k \leq n-1)$, then
\[
\frac{n(2^{k-1}-n)}{(n-2d)^2}
\]
is a positive integer for some $k$.

\item[$(2)$] 

If $C$ has $d^{\perp} \geq 6$, then
\[
\frac{-n^{2}(n-1)(n-2)}{(n-2d)^{2} \left( n^{2}-(4d+3)n+4d^{2}+2 \right)}
\]
is a positive integer.

\item[$(3)$] 

If $C$ has $d^{\perp} \geq 8$, then 
$n$ is written as $n=(m^2+8)/3$ for some $m\in \ZZ$ and 
\[
\frac{n^2 \left(n^2-3 n+2\right)}{6 (3 n-8)}=\frac{\left(m^2+2\right) \left(m^2+5\right) \left(m^2+8\right)^2}{486 m^2}
\]
and
\[
\frac{2 \left(n^4-7 n^3+23 n^2-41 n+24\right)}{3 (3 n-8)},
\]
are positive integers.
\end{enumerate}
\end{Lem}

\begin{proof}

(1)
By Equations (\ref{eq:dim k}) and (\ref{eq:y2}), we have
\begin{align*}
& 2\alpha \left(-d+\binom{n-2d}{2} \right)-(-2\alpha -2 +2^{k})\frac{n}{2}+2 \binom{n}{2}=0 \\
\Leftrightarrow\ & \alpha \left( n-2d+(n-2d)(n-2d-1) \right)=n\cdot2^{k-1}-n^{2}\\
\Leftrightarrow\ & \alpha(n-2d)^{2}=n(2^{k-1}-n)\\
\Leftrightarrow\ & \alpha=\frac{n(2^{k-1}-n)}{(n-2d)^2}.
\end{align*}
Since $\alpha$ is a positive integer, 
\[
\frac{n(2^{k-1}-n)}{(n-2d)^2}
\]
is a positive integer.

(2)
By Equations (\ref{eq:y2}) and (\ref{eq:y4}), we have 
\begin{align*}
\alpha &\left( -\left(\frac{n}{2}-1 \right )d + \left( \frac{n}{2}-1 \right) \binom{n-2d}{2}+2\binom{d}{2}-2d\binom{n-2d}{2}+2\binom{n-2d}{4}\right)\\
&+\left( \frac{n}{2}-1 \right)\binom{n}{2}+2\binom{n}{4}=0.
\end{align*}
Since $\alpha$ is a positive integer, 
\begin{align*}
\alpha=&\frac{\displaystyle\left( \frac{n}{2}-1 \right)\binom{n}{2}+2\binom{n}{4}}
{\displaystyle\left( \frac{n}{2}-1 \right) \left(d-\binom{n-2d}{2} \right)-2\binom{d}{2}+2d\binom{n-2d}{2}-2\binom{n-2d}{4}} \\
=&\frac{-n^{2}(n-1)(n-2)}{(n-2d)^{2} \left( n^{2}-(4d+3)n+4d^{2}+2 \right)}
\end{align*}
is a positive integer.

(3)
By Equations (\ref{eq:y2})--(\ref{eq:y6}), we have 
\begin{align*}
\alpha=\frac{n^2 \left(n^2-3 n+2\right)}{6 (3 n-8)},
\beta=\frac{2 \left(n^4-7 n^3+23 n^2-41 n+24\right)}{3 (3 n-8)},
d=\frac{1}{2} \left(n-\sqrt{3 n-8}\right).
\end{align*}
Since $d$ is a positive integer, 
$3 n-8$ must be a square number. 
Therefore, we have 
\[
n=\frac{m^2 + 8}{3}.
\]
Since $\alpha$ and $\beta$ are positive integers, 
\[
\alpha=\frac{n^2 \left(n^2-3 n+2\right)}{6 (3 n-8)}=\frac{\left(m^2+2\right) \left(m^2+5\right) \left(m^2+8\right)^2}{486 m^2}
\]
and 
\[
\beta=\frac{2 \left(n^4-7 n^3+23 n^2-41 n+24\right)}{3 (3 n-8)}
\]
are positive integers.
\end{proof}

\subsection{Proof of Theorem \ref{thm:main3} (2)}

Here, 
we provide the proof of Theorem \ref{thm:main3} (2). 
\begin{proof}
The harmonic weight enumerator of $f\in \Harm_{t+1}$ is 
\begin{align*}
W_{C,f}
&=ax^{n-d}y^d+bx^{\frac{n}{2}}y^{\frac{n}{2}}+ax^{d}y^{n-d}\\
&=(xy)^{t+1}(ax^{n-d-(t+1)}y^{d-(t+1)}+bx^{\frac{n}{2}-(t+1)}y^{\frac{n}{2}-(t+1)}+ax^{d-(t+1)}y^{n-d-(t+1)}), 
\end{align*}
where $a,b\in \RR$. 
Set
\begin{align*}
Z_{C,f}=ax^{n-d-(t+1)}y^{d-(t+1)}+bx^{\frac{n}{2}-(t+1)}y^{\frac{n}{2}-(t+1)}+ax^{d-(t+1)}y^{n-d-(t+1)}. 
\end{align*}
Then by Theorem \ref{thm: Bachoc iden.}, 
\begin{align*}
Z_{C^{\perp},f}
=&a'(x+y)^{n-d-(t+1)}(x-y)^{d-(t+1)}\\
&+b'(x+y)^{\frac{n}{2}-(t+1)}(x-y)^{\frac{n}{2}-(t+1)}\\
&+a'(x+y)^{d-(t+1)}(x-y)^{n-d-(t+1)}. 
\end{align*}
By $d^{\perp}=4$ or $6$, the coefficient of $x^{n-2(t+1)}$ is zero. 
Hence, we have $2a'+b'=0$. 
Then 
\begin{align}
Z_{C^{\perp},f}
=&a'(x+y)^{n-d-(t+1)}(x-y)^{d-(t+1)} \notag \\
&-2a'(x+y)^{\frac{n}{2}-(t+1)}(x-y)^{\frac{n}{2}-(t+1)} \notag \\
&+a'(x+y)^{d-(t+1)}(x-y)^{n-d-(t+1)} \notag\\
=&a'\big( (x+y)^{n-d-(t+1)}(x-y)^{d-(t+1)} \notag\\
&-2(x+y)^{\frac{n}{2}-(t+1)}(x-y)^{\frac{n}{2}-(t+1)} \notag\\
&+(x+y)^{d-(t+1)}(x-y)^{n-d-(t+1)} \big) \notag \\
=&a'\big( (x^{2}-y^{2})^{d-(t+1)}(x+y)^{n-2d} \label{eqn:2w} \\
&-2(x^{2}-y^{2})^{\frac{n}{2}-(t+1)} \notag\\
&+(x^{2}-y^{2})^{d-(t+1)}(x-y)^{n-2d} \big).  \notag 
\end{align}


Let 
\[
W_{C^\perp,f}=(xy)^{t+1}Z_{C^{\perp},f}=\sum{a_i}x^{n-i}y^i.
\]
By (\ref{eqn:2w}), for some $c\in \RR$
\begin{align*}
a_{2w+t+1}&=c\\
&\times 
\left(\left(\sum_{i=0}^{w}  (-1)^{w-i} \binom{d-(t+1)}{w-i}\binom{n-2d}{2i}\right)+
(-1)^{w+1} \binom{n/2-(t+1)}{w}\right). 
\end{align*}
By Theorem \ref{thm:design}, 
if the equation 
\[ 
\left(\sum_{i=0}^{w}  (-1)^{w-i} \binom{d-(t+1)}{w-i}\binom{n-2d}{2i}\right)+
(-1)^{w+1} \binom{n/2-(t+1)}{w}=0, 
\]
$D^{\perp}_{2w+t+1}$ is a $(t+1)$-design. 
\end{proof}





\subsection{Proof of Corollary \ref{cor:main1}}


Here, 
we show the proof of Corollary \ref{cor:main1}. 
First, we show the following proposition: 
\begin{Prop}\label{thm:1}

Assume that $C$ has $(d^{\perp},t)=(4,1)$, or $(6,3)$.

\begin{enumerate}
\item[$(1)$]
\begin{enumerate}
\item [{\rm (a)}] In the case $(d^{\perp},t)=(4,1)$. 
\begin{enumerate}
\item[{\rm (i)}]
If the equation
\[ n^{2}-(4d+6)n+4d^{2}+32=0 \] is satisfied, then
$D^{\perp}_{6}$ and $D^{\perp}_{n-6}$ are $2$-designs if $n \equiv 0 \mod 4$, 
$D^{\perp}_{6}$ is a $2$-design if $n \equiv 2 \mod 4$.
\item[{\rm (ii)}]
Let $n\equiv 2\pmod{4}$ and $d\equiv 0\pmod{2}$. 
If the equation
\[ n^{2}-(4d+2)n+4d^{2}+8=0 \] is satisfied, then
$D^{\perp}_{n-4}$ is a $2$-design.

\item[{\rm (iii)}]
If the equation
\begin{align*} 
&n^4-(8 d+15) n^3+(12 d (2 d+5)+145) n^2\\
&-2 (2 d (d (8 d+15)+100)+285) n+16 \left(d^4+25 d^2+109\right)=0
\end{align*}
is satisfied, then
$D^{\perp}_{8}$ and $D^{\perp}_{n-8}$ are $2$-designs.

\end{enumerate}
\item [{\rm (b)}] In the case $(d^{\perp},t)=(6,3)$. 
\begin{enumerate}

\item[{\rm (i)}]
If the equation
\[ n^{2}-(4d+6)n+4d^{2}+56=0 \] is satisfied, then
$D^{\perp}_{8}$ and $D^{\perp}_{n-8}$ are $4$-designs if $n \equiv 0 \mod 4$, 
$D^{\perp}_{8}$ is a $4$-design if $n \equiv 2 \mod 4$.
\item[{\rm (ii)}]
Let $n\equiv 2\pmod{4}$ and $d\equiv 0\pmod{2}$. 
If the equation
\[ n^{2}-(4d+2)n+4d^{2}+16=0 \] is satisfied, then
$D^{\perp}_{n-6}$ is a $4$-design.

\item[{\rm (iii)}]
If the equation
\begin{align*} 
&n^4-(8 d+15) n^3+(12 d (2 d+5)+205) n^2\\
&-4 d (d (8 d+15)+160) n-930 n+16 d^2 \left(d^2+40\right)+4744=0
\end{align*}
 is satisfied, then
$D^{\perp}_{10}$ and $D^{\perp}_{n-10}$ are $4$-designs.
\end{enumerate}

\end{enumerate}

\item[$(2)$]  $n-2d=6$
\begin{enumerate}
\item [{\rm (a)}] Assume that $C$ has $(d^{\perp},t)=(4,1)$. If 
\[
u=\frac{3d-9\pm \sqrt{3d+1}}{4}
\]
is a positive number, then $D^{\perp}_{2u+8}$ is a $2$-design.
\item [{\rm (b)}] Assume that $C$ has $(d^{\perp},t)=(6,3)$. If 
\[
u=\frac{3d-15\pm \sqrt{3d-5}}{4}
\]
is a positive number, then $D^{\perp}_{2u+10}$ is a $4$-design.
\end{enumerate}

\item[$(3)$] $n-2d=8$
\begin{enumerate}
\item [{\rm (a)}] Assume that $C$ has $(d^{\perp},t)=(4,1)$. If $d$
is a square number, then $D^{\perp}_{d+4\pm \sqrt{d}}$ is a $2$-design.
\item [{\rm (b)}] Assume that $C$ has $(d^{\perp},t)=(6,3)$. If $d-2$
is a square number, then $D^{\perp}_{d+4\pm \sqrt{d-2}}$ is a $4$-design. 
\end{enumerate}

\end{enumerate}
\end{Prop}
\begin{proof}
\begin{enumerate}
\item[{\rm (1)}]
We give the proof of (a)--(i). 
The other cases can be proven similarly. 
Let 
\[
W_{C^\perp,f}=(xy)^{t+1}Z_{C^{\perp},f}=\sum{a_i}x^{n-i}y^i.
\]
Then for some $c\in \RR$
\begin{align*}
a_{2w+t+1}&=c\\
&\times 
\left(\left(\sum_{i=0}^{w} (-1)^{w-i} \binom{d-(t+1)}{w-i}\binom{n-2d}{2i}\right)+
(-1)^{w+1} \binom{n/2-(t+1)}{w}\right). 
\end{align*}
Let $(d^{\perp},t)=(4,1)$. Then we have 
\begin{align*}
a_{6}={\rm (constant)}\times 
(n-2 d)^2 (n^{2}-(4d+6)n+4d^{2}+32).
\end{align*}
Therefore, if $ n^{2}-(4d+6)n+4d^{2}+32=0$, then 
by Theorem \ref{thm:design}, 
$D^{\perp}_{6}$ is a $2$-design.

\item[{\rm (2)}]
We give the proof of (a). 
Case (b) can be proven similarly.  
Let $n-2d=6$ and $t=1$. 
By (\ref{eqn:2w}),  
\begin{align*}
Z_{C^{\perp},f}
=&a'\big( (x^{2}-y^{2})^{d-2}(x+y)^{6} 
-2(x^{2}-y^{2})^{d+1} 
+(x^{2}-y^{2})^{d-2}(x-y)^{6} \big) \\
=&a'(x^{2}-y^{2})^{d-2}\big(  (x+y)^{6}-2(x^{2}-y^{2})^{3}+(x-y)^{6} \big) \\
=&a'(x^{2}-y^{2})^{d-2}\big( 4y^{6}+24x^{2}y^{4}+36x^{4}y^{2} \big) \\  
=&4a'\sum_{i=0}^{d-2} \binom{d-2}{i} x^{2d-4-2i}(-y^{2})^{i} \big(y^{6}+6x^{2}y^{4}+9x^{4}y^{2} \big). 
\end{align*}
The coefficient of $x^{2d-4-2u}y^{2u+6}$ in $Z_{C^{\perp},f}$ is equal to 
\[ 
4a'\left( \binom{d-2}{u}-6 \binom{d-2}{u+1}+9 \binom{d-2}{u+2} \right).
\]
Then we have
\begin{align*}
&  \binom{d-2}{u}-6 \binom{d-2}{u+1}+9 \binom{d-2}{u+2}=0 \\
\Leftrightarrow\ & 16u^{2}-(24d-72)u+9d^{2}-57d+80=0 \\
\Leftrightarrow\ & u=\frac{3d-9\pm \sqrt{3d+1}}{4}.
\end{align*}
Let 
\[
W_{C^\perp,f}=(xy)^{2}Z_{C^{\perp},f}=\sum{a_i}x^{n-i}y^i.
\]

If 
\[
u=\frac{3d-9\pm \sqrt{3d+1}}{4}
\]
is a positive number, the coefficient of $x^{2d-4-2u}y^{2u+6}$ in $Z_{C^{\perp},f}$ is zero,
then we have $a_{2u+8}=0$. Hence, by Theorem \ref{thm:design}, $D^{\perp}_{2u+8}$ is a $2$-design.

\item[{\rm (3)}]
We give the proof of (a). 
Case (b) can be proven similarly.  
Let $n-2d=8$ and $t=1$.
By (\ref{eqn:2w}),  
\begin{align*}
Z_{C^{\perp},f}
=&a'\big( (x^{2}-y^{2})^{d-2}(x+y)^{8} 
-2(x^{2}-y^{2})^{d+2} 
+(x^{2}-y^{2})^{d-2}(x-y)^{8} \big) \\
=&a'(x^{2}-y^{2})^{d-2}\big(  (x+y)^{8}-2(x^{2}-y^{2})^{4}+(x-y)^{8} \big) \\
=&a'(x^{2}-y^{2})^{d-2}\big( 64x^{2}y^{6}+128x^{4}y^{4}+64x^{6}y^{2} \big) \\  
=&64a'\sum_{i=0}^{d-2} \binom{d-2}{i} x^{2d-4-2i}(-y^{2})^{i} \big(x^{2}y^{6}+2x^{4}y^{4}+x^{6}y^{2} \big). 
\end{align*}
The coefficient of $x^{2d-4-2v}y^{2v+6}$ in $Z_{C^{\perp},f}$ is equal to 
\[ 
64a'\left( \binom{d-2}{v}-2 \binom{d-2}{v+1}+ \binom{d-2}{v+2} \right).
\]
Then we have
\begin{align*}
&  \binom{d-2}{v}-2 \binom{d-2}{v+1}+ \binom{d-2}{v+2}=0 \\
\Leftrightarrow\ & 4v^{2}-(4d-16)v+d^{2}-9d+16=0 \\
\Leftrightarrow\ & v=\frac{d-4\pm \sqrt{d}}{2}.
\end{align*}

If $d$ is a square number,
\[
v=\frac{d-4\pm \sqrt{d}}{2}
\]
is a positive number, hence the coefficient of $x^{2d-4-2v}y^{2v+6}$ in $Z_{C^{\perp},f}$ is zero.
Therefore, by Theorem \ref{thm:design}, 
$D^{\perp}_{2v+8}$ is a $2$-design.
\end{enumerate}
\end{proof}




Now we give the proof of Corollary \ref{cor:main1}. 

\begin{proof}[Proof of Corollary \ref{cor:main1}]

Let $C$ be a binary code of length $n$ and 
minimum distance $d$. 
We assume that $d^{\perp}-t=3$ and $\mathbf{1}_n\in C$. 
Then the weight distribution of $C$ is $0,d,n/2,n-d,n$ and 
$n$ is even. 

For the case $(d^{\perp},t)=(4,1)$ and $n\leq 10000$, 
we give $n$, $d$, and $w$ in Table \ref{tab: B} satisfying the 
conditions in Proposition \ref{thm:1} and Lemma \ref{lem:d^perp=6}.



This completes the proof of Corollary \ref{cor:main1}. 
\end{proof}







\subsection{Proof of Theorem \ref{thm:main4}}
\begin{proof}[Proof of Theorem \ref{thm:main4}]
By Lemma \ref{lem:d^perp=6} (3), 
$n$ is written as $n=(m^2+8)/3$ for some $m\in \ZZ$ and 
\[
\alpha=\frac{\left(m^2+2\right) \left(m^2+5\right) \left(m^2+8\right)^2}{486 m^2}
\]
is a positive integer. Therefore, 
\[
\frac{\left(m^2+2\right) \left(m^2+5\right) \left(m^2+8\right)^2}{m^2}
=m^6+23 m^4+186 m^2+\frac{640}{m^2}+608
\]
is also a positive integer. 
Since $640/m^2$ is a positive integer, we have $m=1,2,4,8$. 
Then, since $n>d^{\perp}=8$, we have $n=24$, and $\alpha=759$, $\beta=2576$. 
Thus $C$ has $d=8$ and $|C|=2^{12}$.
Hence, $C$ is the unique extended Golay code $\mathcal{G}_{24}$.
\end{proof}











\section{Concluding Remarks}\label{sec:rem}

\begin{Rem}
\begin{enumerate}

\item [(1)]

In \cite{MN-tec}, we found examples satisfying the 
condition of Theorem \ref{thm:main3} (2). 
These examples are triply even codes of length $48$, 
which were classified by 
Betsumiya--Munemasa \cite{Betsumiya-Munemasa 2012}. 

This gives rise to a natural question: 
are there other examples satisfying the 
condition of Theorem \ref{thm:main3}?

\item [(2)]

For the case $(d^{\perp},t)=(6,3)$ and $n\leq 10000$, 
there are no values $n$, $d$, and $w$ satisfying the 
conditions in Proposition \ref{thm:1} and Lemma \ref{lem:d^perp=6}.

\item [(3)]
In Corollary \ref{cor:main1}, 
for $n\leq 10000$, 
we give the parameters $n$ and $d$ 
such that $\delta(C)<s(C)$ occurs. 
We remark that for $n\leq 1000$, 
Table \ref{tab: B} gives complete values such that 
$\delta(C)<s(C)$ occurs. 

\item [(4)]

We will discuss the cases $d^{\perp}-t\geq 4$ in the subsequent papers \cite{MN-St-2}.

\end{enumerate}

\end{Rem}

\section*{Acknowledgments}
The authors thank 
Akihiro Munemasa for helpful discussions and computations in this research.
The second author thanks Masaaki Harada and Hiroki Shimakura for their comments 
in a seminar at Tohoku university.
The first author is supported by JSPS KAKENHI (18K03217). 


\appendix 
\section{Values of $X_{ij}$}\label{sec:Xij}

\begin{align*}
X_{11}&=-2 d (d-n) \left((n-2 d)^2-n+2\right).\\
X_{12}&=-\frac{1}{4} (n-2) n^2.\\
X_{21}&=8 d^2 \binom{n}{6}+\frac{1}{12} (n-1) n\\ 
&\Big(-24 \binom{n-2 d}{6}+16 d^5-32 d^4 n+24 d^4+24 d^3 n^2\\
&-48 d^3 n+60 d^3-8 d^2 n^3+30 d^2 n^2-62 d^2 n+12 d^2\\
&+d n^4-6 d n^3+17 d n^2-12 d n+8 d\Big)-8 d n \binom{n}{6}+2 n^2 \binom{n}{6}-2 n \binom{n}{6}.\\
X_{22}&=\frac{1}{48} (n-1) n \left(n^3-6 n^2+8 n\right)-n \binom{n}{6}.\\
X_{31}&=\frac{1}{12} \Big(-16 d^6 \binom{n}{6}+32 d^5 n \binom{n}{6}
-16 d^5 \binom{n}{6}-32 d^5 \binom{n}{8}-24 d^4 n^2 \binom{n}{6}\\
&+24 d^4 n \binom{n}{6}-38 d^4 \binom{n}{6}+64 d^4 n \binom{n}{8}
-48 d^4 \binom{n}{8}+8 d^3 n^3 \binom{n}{6}-8 d^3 n^2 \binom{n}{6}\\
&-48 d^3 n^2 \binom{n}{8}+16 d^3 n \binom{n}{6}+52 d^3 \binom{n}{6}
+96 d^3 n \binom{n}{8}-120 d^3 \binom{n}{8}-d^2 n^4 \binom{n}{6}\\
&-2 d^2 n^3 \binom{n}{6}+16 d^2 n^3 \binom{n}{8}+13 d^2 n^2 \binom{n}{6}
-60 d^2 n^2 \binom{n}{8}-58 d^2 n \binom{n}{6}+6 d^2 \binom{n}{6}\\
&+124 d^2 n \binom{n}{8}-24 d^2 \binom{n}{8}+d n^4 \binom{n}{6}
-2 d n^4 \binom{n}{8}-6 d n^3 \binom{n}{6}+12 d n^3 \binom{n}{8}\\
&+19 d n^2 \binom{n}{6}-34 d n^2 \binom{n}{8}-14 d n \binom{n}{6}
+12 d \binom{n}{6}+24 d n \binom{n}{8}-16 d \binom{n}{8}\\
&+48 d \binom{n}{6} \binom{n-2 d}{6}+48 \binom{n}{8} \binom{n-2 d}{6}
-48 \binom{n}{6} \binom{n-2 d}{8}\Big).\\
X_{32}&=\frac{1}{192} \Big(n^4 (-\binom{n}{6})+12 n^3 \binom{n}{6}
-8 n^3 \binom{n}{8}-44 n^2 \binom{n}{6}+48 n^2 \binom{n}{8}\\
&+48 n \binom{n}{6}-64 n \binom{n}{8}\Big).
\end{align*}

\section{Table of $(d^\perp,t)=(4,1)$}
{\scriptsize
\[
\hspace{-55pt}
\begin{array}{c||c|c|c|c|c|c|c|c|c|c}\label{tab: B}
n &16&22&22&38&40&48&54&72&72&80\\ \hline
d &4&6&8&16&16&16&22&26&33&36\\ \hline
w &6,10&6&18&24&16,24&6,42&50&6,66&48,58&34,46\\ 
\hline\hline
n &86&86&102&118&118&136&136&166&198&208\\ \hline
d &32&40&44&46&56&64&65&74&96&100\\ \hline
w &6&58&98&6&94&60,76&94,108&162&156&94,114\\ 
\hline\hline
n &246&246&272&296&296&328&328&342&358&358\\ \hline
d &104&112&116&120&144&142&161&158&156&176\\ \hline
w &6&242&6,266&8,288&136,160&6,322&234,256&338&6&256\\ 
\hline\hline
n &400&422&422&454&456&520&566&566&582&646\\ \hline
d &196&186&208&212&202&256&254&280&274&320\\ \hline
w &186,214&6&328&450&6,450&246,276&6&438&578&468\\ 
\hline\hline
n &656&726&776&808&886&918&968&968&976&1062\\ \hline
d &324&344&385&400&422&456&446&481&484&508\\ \hline
w &310,346&722&564,598&384,424&882&706&6,962&706,744&466,510&1058\\ 
\hline\hline
n &1072&1126&1160&1238&1254&1296&1360&1416&1462&1478\\ \hline
d &496&560&576&616&602&604&676&662&704&692\\ \hline
w &6,1066&864&556,604&906&1250&6,1290&654,706&6,1410&1458&6\\ 
\hline\hline
n &1478&1576&1606&1672&1672&1686&1808&1808&1926&2056\\ \hline
d &736&784&800&786&833&814&852&900&932&1024\\ \hline
w &1084&760,816&1228&6,1666&1228,1278&1682&6,1802&874,934&1922&996,1060\\ 
\hline\hline
n &2182&2248&2248&2320&2326&2454&2486&2486&2600&2742\\ \hline
d &1058&1066&1121&1156&1160&1192&1182&1240&1296&1334\\ \hline
w &2178&6,2242&1656,1714&1126,1194&1714&2450&6&1894&1264,1336&2738\\ 
\hline\hline
n &2822&2896&2998&3046&3208&3272&3366&3366&3366&3536\\ \hline
d &1408&1444&1496&1484&1600&1633&1612&1642&1680&1764\\ \hline
w &2148&1410,1486&2214&3042&1564,1644&2418,2488&6&3362&2488&1726,1810\\ 
\hline\hline
n &3656&3656&3702&3856&3880&3958&4054&4166&4240&4422\\ \hline
d &1754&1825&1808&1852&1936&1976&1982&2080&2116&2164\\ \hline
w &6,3650&2704,2778&3698&6,3850&1896,1984&3006&4050&3084&2074,2166&4418\\ 
\hline\hline
n &4598&4598&4616&4806&4822&4936&5008&5168&5206&5416\\ \hline
d &2216&2296&2304&2354&2408&2465&2500&2496&2552&2704\\ \hline
w &6&3406&2260,2356&4802&3658&3658,3744&2454,2554&6,5162&5202&2656,2760\\ 
\hline\hline
n &5526&5622&5648&5840&5896&5896&6022&6022&6054&6278\\ \hline
d &2760&2758&2732&2916&2854&2945&2916&3008&2972&3042\\ \hline
w &4098&5618&6,5642&2866,2974&6,5890&4374,4468&6&4468&6050&6\\ 
\hline\hline
n &6278&6280&6408&6502&6672&6736&6806&6806&6966&7078\\ \hline
d &3136&3136&3201&3194&3236&3364&3302&3400&3424&3536\\ \hline
w &4756&3084,3196&4756,4854&6498&6,6666&3310,3426&6&5154&6962&5256\\ 
\hline\hline
n &7208&7216&7446&7496&7496&7638&7696&7942&8072&8200\\ \hline
d &3600&3504&3662&3642&3745&3712&3844&3908&4033&4096\\ \hline
w &3544,3664&6,7210&7442&6,7490&5568,5674&6&3786,3910&7938&5998,6108&4036,4164\\ 
\hline\hline
n &8454&8518&8720&8822&8822&8976&8982&9256&9288&9288\\ \hline
d &4162&4256&4356&4296&4408&4372&4424&4624&4526&4641\\ \hline
w &8450&6444&4294,4426&6&6558&6,8970&8978&4560,4696&6,9282&6906,7024\\ 
\hline\hline
n &9446&9526&9766&9766&9808&9928&9928&&&\\ \hline
d &4720&4694&4762&4880&4900&4842&4961&&&\\ \hline
w &7024&9522&6&7384&4834,4974&6,9922&7384,7506&&&
\end{array}
\]}




\end{document}